\documentclass{article}
\usepackage{amssymb,amsmath,amsthm,graphicx}
\usepackage[all,color]{xy}

\def\qed{\hfill {\hbox{${\vcenter{\vbox{               %HOLLOW SQUARE
   \hrule height 0.4pt\hbox{\vrule width 0.4pt height 6pt
   \kern5pt\vrule width 0.4pt}\hrule height 0.4pt}}}$}}}

\def\utr{\, \underline{\triangleright}\, }
\def\otr{\, \overline{\triangleright}\, }

\newtheorem{theorem}{Theorem}

\newtheorem{proposition}[theorem]{Proposition}
\newtheorem{corollary}[theorem]{Corollary}

\theoremstyle{definition}
\newtheorem{example}{Example}
\newtheorem{definition}{Definition}

\date{}

\title{\Large \textbf{Biquandle Module Quiver Representations}}

\author{Yewon Joung\footnote{Email: yewonjoung@hanyang.ac.kr. Supported by Basic Science Research Program through the National Research Foundation
                          of Korea(NRF) funded by the Ministry of Education(NRF-2022R1I1A1A01070329).}
\and Sam Nelson\footnote{Email: Sam.Nelson@cmc.edu. Partially supported 
by Simons Foundation collaboration grant 702597.}}

\begin{document}
\maketitle

\begin{abstract}
We introduce an infinite family of quiver representation-valued invariants
of classical, virtual and surface-knots and links associated to a choice
of finite biquandle, commutative unital ring, biquandle module and set of 
biquandle endomorphisms.
As an application, we use this quiver to define a new infinite family of 
two-variable polynomial invariants.
\end{abstract}

\parbox{5.5in} {\textsc{Keywords:} Biquandle module quivers, 
Bikei module quivers, counting invariants, surface-links, marked graph diagrams,
Quiver representations

\smallskip

 }

\section{\textbf{Introduction}}\label{I}

Introduced in \cite{FRS} and studied in subsequent works such as \cite{KR}, 
\textit{biquandles} are algebraic structures whose axioms encode the 
Reidemeister moves of classical knot theory. Every oriented classical, virtual 
or surface- knot or link $K$ has a \textit{fundamental biquandle} 
$\mathcal{B}(K)$ whose isomorphism class determines $X$ up to 
reversed-orientation mirror image in the classical case
\cite{EN}. Given a finite biquandle $X$, the set of biquandle homomorphisms 
$\mathrm{Hom}(\mathcal{B}(K),X)$ can be represented concretely by fixing a
presentation associated to a
diagram $D$ of $K$ analogously to fixing bases to represent linear 
transformations as matrices; each biquandle homomorphism 
$f:\mathcal{B}(K)\to X$ is represented by a \textit{biquandle coloring} of
our diagram $D$. Changing a biquandle-colored diagram by Reidemeister moves 
gives us a unique new biquandle-colored diagram representing the same 
biquandle homomorphism, analogously to applying a change-of-basis matrix.

\textit{Biquandle modules} with coefficients in a commutative unital ring $k$
generalize the Alexander module construction to the case of biquandle-colored 
oriented knots and links. More precisely, the Alexander module of a classical 
knot or link is a particular biquandle module with single-element 
coloring biquandle $X=\{1\}$ and coefficient ring $\mathbb{Z}[t^{\pm 1}]$.
Fixing a finite biquandle $X$ and biquandle module $M$, each element of the
biquandle homset determines an $k$-module which invariant under Reidemeister
moves; the multiset of these modules over the homset is the \textit{biquandle
module enhancement} of the counting invariant, previously studied in \cite{BN,CN} etc.

A subset of the set of endomorphisms $\mathrm{Hom}(X,X)$ of a biquandle 
determines a quiver structure on the homset $\mathrm{Hom}(\mathcal{B}(K),X)$.
A choice of biquandle module then gives us a weighting of the vertices in the
quiver, categorifying the biquandle module enhancement from \cite{JN1}. 
In this paper we extend this construction into a full quiver representation 
by defining module homomorphisms associated to the arrows in the quiver. We 
define new polynomial invariants of classical and virtual knots and links as 
well as surface-links from this quiver.

The paper is organized as follows. In Section \ref{BBM} we review the basics
of biquandles and biquandle modules. In Section \ref{BMQR} we recall biquandle 
coloring quivers and biquandle module quivers, introducing our new quiver 
representation and its associated polynomial knot invariant. In Section 
\ref{EC} we collect some examples and computations of the new invariants for
classical and virtual knots and links as well as oriented surface-links. We
conclude in Section \ref{Q} with some questions for future research.

This paper, including all text, diagrams, figures and computational code has 
been produced exclusively by the authors and entirely without the use of any 
form of generative AI.

\section{\textbf{Biquandles, Biquandle Modules and Quivers}}\label{BBM}

We begin with a definition; see \cite{EN} and the references therein for more.

\begin{definition}
A \textit{biquandle} is a set $X$ with two binary operations 
$\utr,\otr:X\times X\to X$ satisfying the following axioms:
\begin{itemize}
\item[(i)] For every $x\in X$ we have $x\utr x=x\otr x$,
\item[(ii)] For all $y\in X$ the maps $\alpha_y,\beta_y:X\to X$ defined by 
$\alpha_y(x)=x\otr y$ and $\beta_y(x)=x\utr y$ and the map 
$S:X\times X \to X\times X$ defined by $S(x,y)=(y\otr x,x\utr y)$
are invertible, and
\item[(iii)] For all $x,y,z\in X$ we have the \textit{exchange laws}
\[\begin{array}{rcl}
(x\utr y)\utr(z\utr y) & = & (x\utr z)\utr(y\otr z) \\
(x\utr y)\otr(z\utr y) & = & (x\otr z)\utr(y\otr z) \\
(x\otr y)\otr(z\otr y) & = & (x\otr z)\otr(y\utr z). 
\end{array}
\]
\end{itemize}
%If in addition to axiom (ii) we also have
%\begin{itemize}
%\item[(ii$'$)] For all $x,y$in $X$ we have
%\[\begin{array}{rcl}
%(x\utr y)\utr y & = & x \\
%(x\otr y)\otr y & = & x \\
%x\utr (y\otr x) & = & x\utr y \\
%x\otr (y\utr x) & = & x\otr y \\
%\end{array}\]
%then we say $X$ is a \textit{bikei}.
%\end{itemize}
A map $\sigma:X\to Y$ between biquandles is a \textit{biquandle 
homomorphism} if for all $x,y\in X$ we have 
\[\begin{array}{rcl}
\sigma(x\utr y) & = & \sigma(x)\utr \sigma(y) \\
\sigma(x\otr y) & = & \sigma(x)\otr \sigma(y) 
\end{array}.\]
A self-homomorphism is an \textit{endomorphism}. 
\end{definition}

\begin{example}
Any set $X$ with choice of bijection $\tau:X\to X$
is a biquandle with operations $x\utr y= \tau(x) =x\otr y$
known as a \textit{constant action biquandle}.
\end{example}

\begin{example}
A group $G$ is a biquandle under the operations
\[x\utr y=y^{-1}xy^{-1}\quad x\otr y = y^{-1}.\]
\end{example}

\begin{example}
A module over $\mathbb{Z}[t^{\pm 1},s^{\pm 1}]$ is a biquandle
(called an \textit{Alexander biquandle}) under the operations
\[x\utr y=tx+(s-t) y\quad x\otr y= sx.\]
\end{example}

\begin{example}
We can specify a biquandle structure on a finite set $X=\{1,\dots,n\}$
by listing the operation tables for $\utr$ and $\otr$. For example, the
smallest nontrivial biquandle has two elements and can be specified by
\[
\begin{array}{r|rr} \utr & 1 & 2 \\ \hline 1 & 2 & 2 \\ 2 & 1& 1\end{array} 
\quad
\begin{array}{r|rr} \otr & 1 & 2 \\ \hline 1 & 2 & 2 \\ 2 & 1& 1\end{array}
\]
or as $\mathbb{Z}_2$ with $x\utr y=x\otr y=x+1$ where we write the class
of zero as $2$.
\end{example}

\begin{definition}
Let $L$ be an oriented classical or virtual knot or link or surface-link
represented by an oriented classical or virtual knot or link diagram or
oriented marked graph diagram $D$. Let $E$ be a set of generators in 
one-to-one correspondence with semiarcs in $D$. The \textit{fundamental 
biquandle} of $L$, denoted $\mathcal{B}(L)$, has presentation with generators 
given by $E$ and relations at the classical crossings given by
\[
\raisebox{-0.65in}{\scalebox{0.85}{\includegraphics{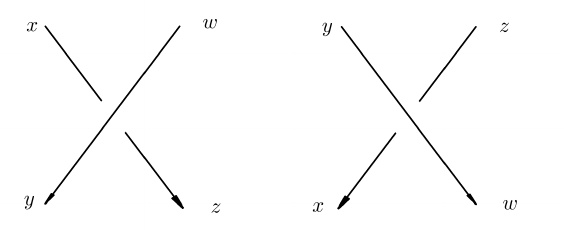}}}\quad
\begin{array}{rcl}
w & = & y\otr x \\
z & = & x\utr y
\end{array}\]
and all four semiarcs meeting at a marked vertex are equivalent.
The elements of the fundamental biquandle are equivalence classes of
biquandle words in these generators (and expressions like 
$S^{-1}_1(x,y)$, $x\utr^{-1}y$, etc. required by axiom (ii)) 
modulo the equivalence relation generated by the biquandle axioms and the 
crossing relations.
\end{definition}

We have the following standard result:
\begin{theorem}
The isomorphism class of the fundamental biquandle is an invariant of
oriented classical knots, virtual knots and surface-links.
\end{theorem}

\begin{proof}
(Sketch) The reader is invited to verify that the biquandle axioms are chosen 
so that Reidemeister moves and Yoshikawa moves on diagrams induce Tietze 
moves on presentations.
\end{proof}

\begin{definition}
Let $X$ be a finite biquandle and $L$ an oriented classical or virtual knot 
or link or surface-link represented by a choice of oriented classical or 
virtual knot or link diagram or oriented marked-graph diagram $D$. A 
\textit{biquandle coloring} or \textit{$X$-coloring} of $D$ is an 
assignment of an element of $X$ to each semiarc in $D$ satisfying
the \textit{coloring condition}.
\[\scalebox{0.85}{
\includegraphics{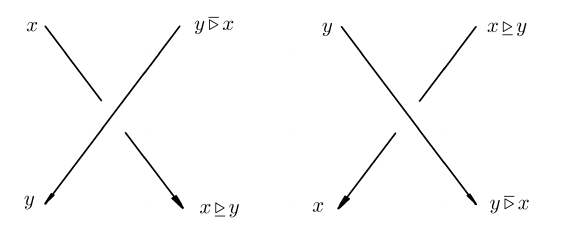}\includegraphics{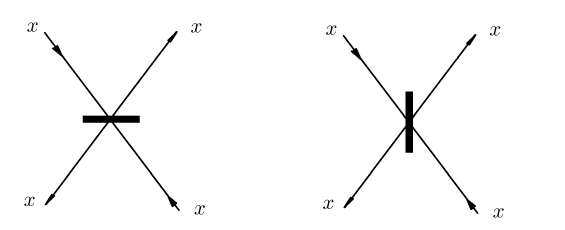}}\]
\end{definition}

A biquandle coloring defines a homomorphism $f:\mathcal{B}(L)\to X$.
The set of these homomorphisms, $\mathrm{Hom}(\mathcal{B}(L),X)$, is called
the \textit{biquandle homset}. The homset can be represented visually
as the set of $X$-colorings of any choice of diagram of $L$.

\begin{example}\label{ex:hs}
The figure 8 knot $4_1$ has three colorings by the biquandle $X$ with operation
tables
\[
\begin{array}{r|rrr}
\utr & 1 & 2 & 3  \\ \hline
1 & 2 & 3 & 1 \\
2 & 3 & 1 & 2 \\
3 & 1 & 2 & 3
\end{array}
\quad
\begin{array}{r|rrr}
\otr & 1 & 2 & 3  \\ \hline
1 & 2 & 2 & 2 \\
2 & 1 & 1 & 1 \\
3 & 3 & 3 & 3
\end{array}
\]
as shown:
\[\includegraphics{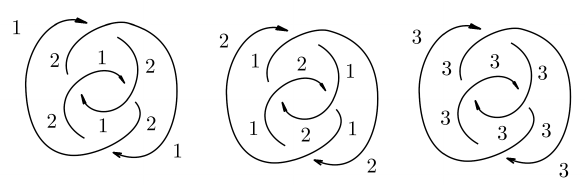}\]
\end{example}

\begin{definition}
Let $X$ be a finite biquandle and $L$ an oriented classical or virtual knot 
or link or surface-link represented by a choice of oriented classical or 
virtual knot or link diagram or oriented marked-graph diagram $D$. Let 
$k$ be a commutative unital ring. A \textit{biquandle module} structure
consists of three maps $t,s,r:X\times X\to k$ such that
\begin{itemize}
\item For all $x\in X$, $t_{x,x}+s_{x,x}=r_{x,x}$,
\item For all $x,y$, the elements $t_{x,y}$ and $r_{x,y}$ are units in $k$ and
\item For all $x,y,z\in X$, we have
\[\begin{array}{rcl}
r_{y\otr x,z\otr x} r_{x,z} & = & r_{x\utr y, z\otr y}r_{y,z}  \\
r_{x\utr z,y\utr z} t_{y,z} & = & t_{y\otr x, z\otr x}r_{x,y}  \\
r_{x\utr z,y\utr z} s_{y,z} & = & s_{y\otr x, z\otr x}r_{x,z}  \\
t_{x\utr z,y\utr z} t_{x,z} & = & t_{x\utr y, z\otr y}t_{x,y}  \\
s_{x\utr z,y\utr z} t_{y,z} & = & t_{x\utr y, z\otr y}s_{x,y}  \\
t_{x\utr z,y\utr z} s_{x,z} +s_{x\utr z, y\utr z}s_{y,z} & = & s_{x\utr y, z\otr y}r_{y,z} 
\end{array}.\]
\end{itemize}
\end{definition}

\begin{example}
The biquandle with operation tables
\[
\begin{array}{r|rrr}
\utr & 1 & 2 & 3 \\ \hline
1 & 2 & 3 & 1 \\
2 & 3 & 1 & 2 \\
3 & 1 & 2 & 3
\end{array}
\quad
\begin{array}{r|rrr}
\otr & 1 & 2 & 3 \\ \hline
1 & 2 & 2 & 2 \\
2 & 1 & 1 & 1 \\
3 & 3 & 3 & 3
\end{array}
\]
has biquandle module structures with $k=\mathbb{Z}_3$
including
\[
\begin{array}{r|rrr}
t & 1 & 2 & 3 \\ \hline
1 & 2 & 1 & 1 \\
2 & 2 & 2 & 1 \\
3 & 1 & 2 & 1 
\end{array}\quad
\begin{array}{r|rrr}
s & 1 & 2 & 3 \\ \hline
1 & 2 & 2 & 1 \\
2 & 1 & 2 & 2 \\
3 & 1 & 1 & 1 
\end{array}\quad
\begin{array}{r|rrr}
r & 1 & 2 & 3 \\ \hline
1 & 1 & 1 & 2 \\
2 & 1 & 1 & 2 \\
3 & 1 & 1 & 2
\end{array}.
\]
\end{example}

To each element $v$ of $\mathrm{Hom}(\mathcal{B}(L),X)$, a biquandle module 
associates an invariant $k$-module $M_v$ whose elements can be visualized as 
\textit{bead colorings} of an $X$-colored diagram representing the homset 
element.
\[\scalebox{0.85}{\includegraphics{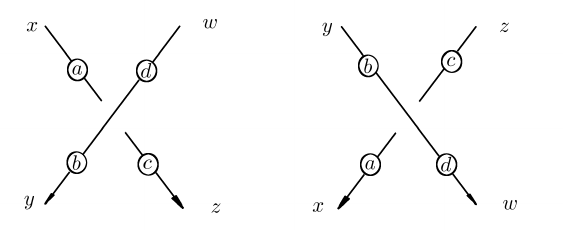} \includegraphics{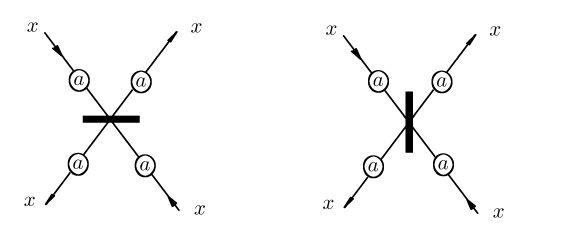}}
\]
where we have
\[c=t_{x,y}a+s_{x.y} b \quad d=r_{x,y} b.\]

Each homset element has a $k$-module of bead colorings which is invariant
up to isomorphism under Reidemeister moves in the classical case,
virtual Reidemeister moves in the virtual case, and Yoshikawa moves in the
oriented surface-link case.

\begin{example}
The homset element from Example \ref{ex:hs}
\[\includegraphics{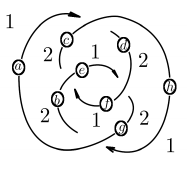}\] 
has bead-coloring matrix
\[
\left[\begin{array}{rrrrrrrr}
2 & 0 & 2 & 2 & 0 & 0 & 0 & 0  \\
0 & 0 & 1 & 0 & 0 & 0 & 0 & 2  \\
0 & 0 & 2 & 0 & 2 & 2 & 0 & 0  \\
0 & 2 & 0 & 0 & 1 & 0 & 0 & 0  \\
0 & 0 & 0 & 0 & 2 & 2 & 2 & 0  \\
0 & 0 & 0 & 2 & 0 & 1 & 0 & 0  \\
0 & 2 & 0 & 0 & 0 & 0 & 2 & 2  \\
2 & 0 & 0 & 0 & 0 & 0 & 1 & 0  \\
\end{array}\right]
\sim
\left[\begin{array}{rrrrrrrr}
1 & 0 & 0 & 0 & 0 & 0 & 0 & 2  \\
0 & 1 & 0 & 0 & 0 & 0 & 0 & 2  \\
0 & 0 & 1 & 0 & 0 & 0 & 0 & 2  \\
0 & 0 & 0 & 1 & 0 & 0 & 0 & 2  \\
0 & 0 & 0 & 0 & 1 & 0 & 0 & 2  \\
0 & 0 & 0 & 0 & 0 & 1 & 0 & 2  \\
0 & 0 & 0 & 0 & 0 & 0 & 1 & 2  \\
0 & 0 & 0 & 0 & 0 & 0 & 0 & 0  \\
\end{array}\right]
\]
and hence bead-coloring module $(\mathbb{Z}_3)^1$.
\end{example}

%Let $X$ be a finite biquandle, $D$ an $X$-colored oriented classical or 
%virtual knot or link diagram and  $k$ be a commutative unital ring. 

\begin{definition}
Let $D$ be an $X$-colored diagram and $f:X\to X$ a biquandle endomorphism. 
Then applying $f$ to each of the colors in $D$ results in another (not 
necessarily distinct) element of the homset; hence, as observed in \cite{ChN},
any subset $S$ of the set of biquandle endomorphisms determines an invariant 
quiver structure on the homset, known as the \textit{biquandle coloring quiver}
of the link $L$ represented by $D$ with respect to $(X,S)$. If 
$S=\mathrm{End}(X)$ then $BCQ(L,X)$ is the \textit{full quiver}.
\end{definition}

\section{\textbf{Biquandle Module Quivers and Representations}}\label{BMQR}

In \cite{IN}, the biquandle module enhancement was enhanced with 
the quiver structure, providing a categorification of the biquandle 
and bikei module invariant as quivers are categories. Often the next
step in categorification is to go from quivers to \textit{quiver representations}
i.e., quivers with modules at the vertices and linear transformations along 
the arrows. We will now introduce a quiver representation-valued invariant
of oriented classical and virtual knots and links and oriented surface-links.

Recall that the \textit{image subbiquandle} $\mathrm{Im}(D_v)$ of an 
$X$-colored diagram
$D_v$ is the closure of the set of elements of $X$ appearing a semiarc
labels in $D$; equivalently, it is the image of the coloring considered
as a biquandle homomorphism from $v:\mathcal(B)(L) \to X$. Then the
key observation is that if the biquandle module coefficients don't 
change when we apply the endomorphism $\sigma$, the bead coloring equations
don't change and the bead coloring spaces are naturally isomorphic. We can 
then define a quiver representation by assigning the identity map to arrows
satisfying this condition and assigning the zero map otherwise. More 
formally, we have:

\begin{definition}
Let $X$ be a finite biquandle and $L$ an oriented classical knot or link 
(respectively, virtual knot or link or surface-link) represented by a choice 
of classical knot or link diagram (respectively, virtual knot or link diagram 
or marked-graph diagram) $D$. Let $S\subset \mathrm{End}(X)$ be a subset of
the set of biquandle endomorphisms of $X$, $M$ a $X$-module with coefficients 
in a commutative unital ring $k$. Then the \textit{biquandle module quiver 
representation} of $L$ with respect to the \textit{data vector} $(X,M,k,S)$
is obtained from the biquandle coloring quiver of $L$ with respect to $(X,S)$
by weighting each vertex with the corresponding module of bead-colorings
and each arrow with the linear transformation $\phi_{\sigma}$ where
\[\phi_{\sigma}=\left\{
\begin{array}{ll}
\mathrm{Id} & t_{xy}=t_{\sigma(x)\sigma(y)}, s_{xy}=s_{\sigma(x)\sigma(y)}\ 
\mathrm{and}\ r_{xy}=r_{\sigma(x)\sigma(y)} \ \forall x,y\in \mathrm{Im}(D_v) \\
0 & \mathrm{otherwise}
\end{array}.
\right.\]
\end{definition}

\begin{proposition}
The biquandle module quiver representation is an invariant of oriented 
classical links, oriented virtual links and orientable surface-links.
\end{proposition}

\begin{proof}
The biquandle module quiver is known (indeed, constructed) to be invariant
under Reidemeister/virtual Reidemeister/Yoshikawa moves. Then it suffices to
observe that changing the diagram by such moves induces the same change of 
basis on all of the bead-coloring matrices, and hence if two bead-coloring 
matrices were equal before a move, they are equal after the move.
\end{proof}

\begin{example}\label{ex:1}
Let $X$ be the biquandle given by the operation tables
\[
\begin{array}{r|rrr}
\utr & 1 & 2 & 3 \\ \hline
1 & 2 & 2 & 2 \\
2 & 1 & 1 & 1 \\
3 & 3 & 3 & 3
\end{array}
\quad
\begin{array}{r|rrr}
\utr & 1 & 2 & 3 \\ \hline
1 & 2 & 3 & 1 \\
2 & 3 & 1 & 2 \\
3 & 1 & 2 & 3
\end{array}.
\]
Then we observe that the tables
\[
\begin{array}{r|rrr}
t & 1 & 2 & 3 \\ \hline
1 & 1 & 1 & 1 \\
2 & 1 & 1 & 1 \\
3 & 1 & 1 & 1
\end{array}
\quad
\begin{array}{r|rrr}
s & 1 & 2 & 3 \\ \hline
1 & 1 & 1 & 2 \\
2 & 1 & 1 & 2 \\
3 & 2 & 2 & 1
\end{array}
\quad
\begin{array}{r|rrr}
r & 1 & 2 & 3 \\ \hline
1 & 2 & 1 & 1 \\
2 & 1 & 2 & 1 \\
3 & 2 & 2 & 2
\end{array}
\]
define an $X$-module over $\mathbb{Z}_3$ and that the maps
\[\begin{array}{r|rrr}
x & 1 & 2 & 3 \\ \hline
\sigma_1(x) & 3 & 3 & 3 \\
\sigma_2(x) & 1 & 2 & 3 \\
\sigma_3(x) & 2 & 1 & 3 \\
%1 & 3 & 1 & 2 \\
%2 & 3 & 2 & 1 \\
%3 & 3 & 3 & 3
\end{array}\]
form the complete set of endomorphisms of $X$. Then the complete  biquandle 
coloring quiver of the the $4_1$ knot is
\[\includegraphics{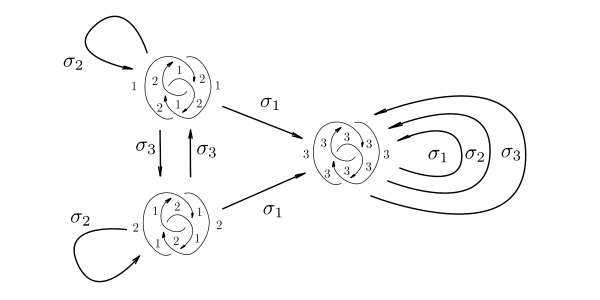}.\]
We compute that each of the three colorings has a 1-dimensional space of bead 
colorings; then the biquandle module quiver representation is
\[\includegraphics{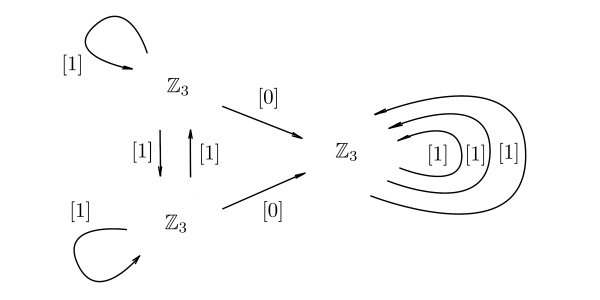}.\]
\end{example}

Comparing quiver representations directly can be computationally intensive
for large quivers, so we find it convenient to define a polynomial invariant
from the quiver representation.

\begin{definition}
Let $X$ be a finite biquandle and $L$ an oriented classical knot or link, 
virtual knot or link or oriented surface-link, $S\subset \mathrm{End}(X)$ 
a subset of the set of biquandle endomorphisms of $X$ and $M$ a $X$-module 
with coefficients in a commutative unital ring $k$. In the resulting biquandle 
module quiver, let $MP$ be the set of all maximal-length non-repeating
paths in which every edge's associated matrix is an identity matrix. We then 
define the \textit{natural path polynomial} of $L$ with respect to the data 
vector $\vec{D}=(X,M,k,S)$ to be the sum over paths $p\in MP$ of terms of the 
form $x^{\mathrm{rank}(M_v)}y^{|p|}$ where $|p|$ is the length of the path $p$, i.e.,
\[\Phi_{\vec{D}}^{MP}(L)=\sum_{p\in MP} x^{\mathrm{rank}(M_v)}y^{|p|}.\]
\end{definition}

We then have:

\begin{corollary}
The natural path polynomial $\Phi_{\vec{D}}^{MP}(L)$ is an invariant of
oriented classical and virtual links and oriented surface-links.
\end{corollary}

\begin{example}
In the biquandle module quiver representation in Example \ref{ex:1}
we have natural path polynomial $\Phi_{\vec{D}}^{MP}(4_1)=4xy^4+6xy^3$.
\end{example}

\section{\textbf{Examples and Computations}}\label{EC}

In this Section we collect some examples and computations. We stress that
these are toy examples, selected because their small size makes them easily 
computable via \texttt{python} code. We remark that the true power of this 
infinite family of invariants lies in the choice of larger and more complex
biquandles, modules over larger finite or infinite rings, and larger sets of 
endomorphisms.

\begin{example}
Let $X$ be the biquandle with operation table
\[
\begin{array}{r|rrrr}
\utr & 1 & 2 & 3 & 4 \\ \hline
1 & 2 & 2 & 1 & 2\\
2 & 1 & 1 & 2 & 1\\
3 & 3 & 3 & 4 & 4\\
4 & 4 & 4 & 3 & 3
\end{array}
\quad
\begin{array}{r|rrrr}
\otr & 1 & 2 & 3 & 4 \\ \hline
1 & 2 & 2 & 1 & 2\\
2 & 1 & 1 & 2 & 1\\
3 & 3 & 3 & 4 & 4\\
4 & 4 & 4 & 3 & 3
\end{array}
\]
Via \texttt{python} code, we compute that $X$ has biquandle modules over 
$\mathbb{Z}_3$ including
\[
\begin{array}{r|rrrr}
t & 1 & 2 & 3 & 4 \\ \hline
1 & 1 & 1 & 1 & 1 \\
2 & 1 & 1 & 1 & 1 \\
3 & 1 & 1 & 2 & 1 \\
4 & 1 & 1 & 2 & 1
\end{array}\quad
\begin{array}{r|rrrr}
s & 1 & 2 & 3 & 4 \\ \hline
1 & 0 & 0 & 0 & 0 \\
2 & 0 & 0 & 0 & 0 \\
3 & 0 & 0 & 2 & 2 \\
4 & 0 & 0 & 1 & 1
\end{array}\quad
\begin{array}{r|rrrr}
r & 1 & 2 & 3 & 4 \\ \hline
1 & 1 & 1 & 1 & 1 \\
2 & 1 & 1 & 1 & 1 \\
3 & 2 & 2 & 1 & 1 \\
4 & 2 & 2 & 2 & 2
\end{array}
\]
and endomorphisms including the map
\[
\begin{array}{r|rrrr}
x & 1 & 2 & 3 & 4 \\ \hline
\sigma(x) & 2 & 1 & 3 & 4
\end{array}.
\]
Then we compute the natural path polynomials of the prime classical links with
up to seven crossings in the table at \cite{KA} as shown
in the table.
\[
\begin{array}{r|l}
L & \Phi_{\vec{D}}^{MP}(L) \\ \hline
L2a1 & 4x^2y^2 + 8xy^2 + 4xy \\
L4a1 & 12x^2y^2 + 4xy \\
L5a1 & 12x^2y^2 + 4xy \\
L6a1 & 12x^2y^2 + 4x^2y \\
L6a2 & 4x^2y^2 + 8xy^2 + 4xy \\
L6a3 & 4x^2y^2 + 4x^2y + 8xy^2 \\
L6a4 & 56x^3y^2 + 8y \\
L6a5 & 8x^3y^2 + 24x^2y^2 + 8x^2y + 24xy^2 \\
L6n1 & 8x^3y^2 + 24x^2y^2 + 24xy^2 + 8xy \\
\end{array}
\quad
\begin{array}{r|l}
L & \Phi_{\vec{D}}^{MP}(L) \\ \hline
L7a1 & 12x^2y^2 + 4x^2y \\
L7a2 & 4x^3y^2 + 8x^2y^2 + 4xy \\
L7a3 & 4x^3y^2 + 8x^2y^2 + 4xy \\
L7a4 & 12x^2y^2 + 4xy \\
L7a5 & 4x^2y^2 + 4x^2y + 8xy^2 \\
L7a6 & 4x^2y^2 + 8xy^2 + 4xy \\
L7a7 & 8x^3y^2 + 24x^2y^2 + 24xy^2 + 8xy \\
L7n1 & 4x^3y^2 + 8x^2y^2 + 4xy \\
L7n2 & 4x^3y^2 + 8x^2y^2 + 4xy \\
\end{array}
\]
%[[(2, 0, 1), 4*x**2*y**2 + 8*x*y**2 + 4*x*y], [(4, 0, 1), 12*x**2*y**2 + 4*x*y], [(5, 0, 1), 12*x**2*y**2 + 4*x*y], [(6, 0, 1), 12*x**2*y**2 + 4*x**2*y], [(6, 0, 2), 4*x**2*y**2 + 8*x*y**2 + 4*x*y], [(6, 0, 3), 4*x**2*y**2 + 4*x**2*y + 8*x*y**2], [(6, 0, 4), 56*x**3*y**2 + 8*y], [(6, 0, 5), 8*x**3*y**2 + 24*x**2*y**2 + 8*x**2*y + 24*x*y**2], [(6, 1, 1), 8*x**3*y**2 + 24*x**2*y**2 + 24*x*y**2 + 8*x*y], [(7, 0, 1), 12*x**2*y**2 + 4*x**2*y], [(7, 0, 2), 4*x**3*y**2 + 8*x**2*y**2 + 4*x*y], [(7, 0, 3), 4*x**3*y**2 + 8*x**2*y**2 + 4*x*y], [(7, 0, 4), 12*x**2*y**2 + 4*x*y], [(7, 0, 5), 4*x**2*y**2 + 4*x**2*y + 8*x*y**2], [(7, 0, 6), 4*x**2*y**2 + 8*x*y**2 + 4*x*y], [(7, 0, 7), 8*x**3*y**2 + 24*x**2*y**2 + 24*x*y**2 + 8*x*y], [(7, 1, 1), 4*x**3*y**2 + 8*x**2*y**2 + 4*x*y], [(7, 1, 2), 4*x**3*y**2 + 8*x**2*y**2 + 4*x*y]]
In particular, we note that this example shows that $\Phi_{\vec{D}}^{MP}(L)$
is not determined by the biquandle counting invariant since both $L2a1$ and 
$L4a1$ have counting invariant value 16 with respect to $X$.
\end{example}  

\begin{example}
Let $X$ be the biquandle with operation tables
\[
\begin{array}{r|rrrr}
\utr & 1 & 2 & 3 & 4 \\ \hline
1 & 2 & 2 & 2 & 2 \\
2 & 1 & 1 & 1 & 1 \\
3 & 3 & 3 & 4 & 4 \\
4 & 4 & 4 & 3 & 3
\end{array}
\quad 
\begin{array}{r|rrrr}
\otr & 1 & 2 & 3 & 4 \\ \hline
1 & 2 & 2 & 1 & 1 \\
2 & 1 & 1 & 2 & 2 \\
3 & 4 & 4 & 4 & 4 \\
4 & 3 & 3 & 3 & 3
\end{array}.
\]
We compute that $X$ has biquandle modules over $\mathbb{Z}_3$ including
\[
\begin{array}{r|rrrr}
t & 1 & 2 & 3 & 4 \\ \hline
1 & 1 & 1 & 2 & 2 \\
2 & 1 & 1 & 2 & 2 \\
3 & 1 & 1 & 1 & 1 \\
4 & 1 & 1 & 1 & 1 
\end{array}
\quad 
\begin{array}{r|rrrr}
s & 1 & 2 & 3 & 4 \\ \hline
1 & 1 & 1 & 2 & 1 \\
2 & 1 & 1 & 2 & 1 \\
3 & 1 & 1 & 1 & 2 \\
4 & 2 & 2 & 2 & 1
\end{array}
\quad 
\begin{array}{r|rrrr}
r & 1 & 2 & 3 & 4 \\ \hline
1 & 2 & 2 & 1 & 1 \\
2 & 2 & 2 & 1 & 1 \\
3 & 1 & 1 & 2 & 2 \\
4 & 1 & 1 & 2 & 2
\end{array}
\]
and has endomorphisms including
\[
\begin{array}{r|rrrr}
x & 1 & 2 & 3 & 4 \\ \hline
\sigma_1(x) & 1 & 2 & 4 & 3 \\
\sigma_2(x) & 2 & 1 & 4 & 3 \\
\sigma_3(x) & 2 & 1 & 3 & 4 \\
\end{array}.
\]
We then compute the table of natural path polynomial values for
prime classical links with up to seven crossings as shown.
\[
\begin{array}{r|l}
L & \Phi_{\vec{D}}^{MP}(L) \\ \hline
L2a1 & 192xy^6 \\
L4a1 & 192xy^6 + 8xy^2 \\
L5a1 & 192xy^6 + 8xy^2 \\
L6a1 & 192x^2y^6 + 8x^2y^2 \\
L6a2 & 192xy^6 \\
L6a3 & 192x^2y^6 \\
L6a4 & 384xy^6 + 16xy^2 + 32y^2 \\
L6a5 & 384x^2y^6 \\
L6n1 & 384xy^6
\end{array}\quad
\begin{array}{r|l}
L & \Phi_{\vec{D}}^{MP}(L) \\ \hline
L7a1 & 192x^2y^6 + 8x^2y^2 \\
L7a2 & 192xy^6 + 8xy^2 \\
L7a3 & 192xy^6 + 8xy^2 \\
L7a4 & 192xy^6 + 8xy^2 \\
L7a5 & 192x^2y^6 \\
L7a6 & 192xy^6 \\
L7a7 & 384xy^6 \\
L7n1 & 192xy^6 + 8xy^2 \\
L7n2 & 192xy^6 + 8xy^2
\end{array}.
%[[(2, 0, 1), 192*x*y**6], [(4, 0, 1), 192*x*y**6 + 8*x*y**2], [(5, 0, 1), 192*x*y**6 + 8*x*y**2], [(6, 0, 1), 192*x**2*y**6 + 8*x**2*y**2], [(6, 0, 2), 192*x*y**6], [(6, 0, 3), 192*x**2*y**6], [(6, 0, 4), 384*x*y**6 + 16*x*y**2 + 32*y**2], [(6, 0, 5), 384*x**2*y**6], [(6, 1, 1), 384*x*y**6], [(7, 0, 1), 192*x**2*y**6 + 8*x**2*y**2], [(7, 0, 2), 192*x*y**6 + 8*x*y**2], [(7, 0, 3), 192*x*y**6 + 8*x*y**2], [(7, 0, 4), 192*x*y**6 + 8*x*y**2], [(7, 0, 5), 192*x**2*y**6], [(7, 0, 6), 192*x*y**6], [(7, 0, 7), 384*x*y**6], [(7, 1, 1), 192*x*y**6 + 8*x*y**2], [(7, 1, 2), 192*x*y**6 + 8*x*y**2]]
\]
In particular this example shows that the natural path polynomial is not 
determined by the original biquandle module polynomial invariant since 
the links $L7a7$ and $L7n1$ both have biquandle module polynomial value
$16u$ but are distinguished by their natural path polynomials in the table.
\end{example}

\begin{example}
Biquandle module quivers and natural path polynomials are defined for oriented
virtual knots and links as well. Let $X$ be the biquandle with operation tables
\[
\begin{array}{r|rrr}
\utr & 1 & 2 & 3 \\ \hline
1 & 2 & 2 & 2 \\
2 & 1 & 1 & 1 \\
3 & 3 & 3 & 3\\
\end{array}
\quad
\begin{array}{r|rrr}
\otr & 1 & 2 & 3 \\ \hline
1 & 2 & 3 & 1 \\
2 & 3 & 1 & 2 \\
3 & 1 & 2 & 3\\
\end{array}
\]
and let $R=\mathbb{Z}_5$; then $X$ has endomorphism set 
\[
\begin{array}{r|rrr}
x & 1 & 2 & 3  \\ \hline
\sigma_1(x) & 1 & 2 & 3 \\
\sigma_2(x) & 2 & 1 & 3 \\
\sigma_3(x) & 3 & 3 & 3  \\
\end{array}
\]
and biquandle modules including
\[
\begin{array}{r|rrr}
t & 1 & 2 & 3 \\ \hline
1 & 1 & 1 & 1 \\
2 & 1 & 1 & 1 \\
3 & 4 & 4 & 4\\ 
\end{array}
\quad
\begin{array}{r|rrr}
s & 1 & 2 & 3 \\ \hline
1 & 1 & 4 & 1 \\
2 & 4 & 1 & 4 \\
3 & 4 & 1 & 4
\end{array}
\quad 
\begin{array}{r|rrr}
r & 1 & 2 & 3 \\ \hline
1 & 2 & 2 & 3 \\
2 & 3 & 2 & 2 \\
3 & 3 & 3 & 3
\end{array}
\]
over $R$. We then compute (via \texttt{python}) the values of the natural path
polynomial for each of the prime virtual knots with up to 4 classical crossings
in the table at \cite{KA}:
\[
\begin{array}{r|l}
\Phi_{\vec{D}}^{MP}(L) & L \\ \hline
2y+6y^3 & 2.1, 3.1, 3.2, 3.3, 3.4, 4.1, 4.2, 4.3, 4.4, 4.6, 4.9. 4.10, 4.11, 4.12, 4.13, 4.14, 4.15, 4.18, 4.20, 4.22, \\ 
       & 4.25, 4.26, 4.27, 4.28, 4.29, 4.30, 4.31, 4.32, 4.33, 4.34, 4.37, 4.38, 4.39, 4.40, 4.43, 4.44, 4.45, 4.46, \\
       & 4.48, 4.49, 4.50, 4.51, 4.52, 4.53, 4.53, 4.69, 4.70, 4.73, 4.74, 4.75, 4.78, 4.81, 4.82, 4.83, 4.84, 4.87, \\
       & 4.88, 4.92, 4.93, 4.94, 4.95, 4.101, 4.103, 4.104 \\
8y+6y^3 & 4.61, 4.62, 4.64 \\
2xy+6xy^3 & 3.5, 4.5, 4.7, 4.8, 4.16, 4.17, 4.19, 4.21, 4.23, 4.24, 4.35, 4.36, 4.41, 4.42, 4.47, 4.55, 4.56, 4.57, \\
          & 4.58, 4.59,  4.60, 4.63, 4.71, 4.72, 4.76, 4.77, 4.79, 4.80, 4.85, 4.86,4.89, 4.90, 4.91, 4.96, 4.97, \\ 
          & 4.100, 4.102, 4.105, 4.106, 4.107, 4.108 \\
8xy+6xy^3 & 3.6, 3.7, 4.65, 4.66, 4.66, 4.67, 4.68, 4.98 \\
6x^2y^y+8x^2y & 4.99 \\
\end{array}
\]
\end{example}

\begin{example}
Biquandle module quivers and natural path polynomials are also 
defined for oriented and unoriented surface-links. Let $X$ be the 
biquandle given by the operation tables
\[
\begin{array}{r|rrr}
\utr & 1 & 2 & 3 \\ \hline
1 & 3 & 1 & 3 \\
2 & 2 & 2 & 2 \\
3 & 1 & 3 & 1\\
\end{array}
\quad
\begin{array}{r|rrr}
\otr & 1 & 2 & 3 \\ \hline
1 & 3 & 3 & 3 \\
2 & 2 & 2 & 2 \\
3 & 1 & 1 & 1\\
\end{array}
\]
and let $R=\mathbb{Z}_3$; then $X$ has endomorphism set 
\[
\begin{array}{r|rrr}
x & 1 & 2 & 3  \\ \hline
\sigma_1(x) & 1 & 2 & 3 \\
\sigma_2(x) & 2 & 2 & 2 \\
\sigma_3(x) & 3 & 2 & 1  \\
\end{array}
\]
and biquandle modules including
\[
\begin{array}{r|rrr}
t & 1 & 2 & 3 \\ \hline
1 & 1 & 1 & 1 \\
2 & 2 & 1 & 2 \\
3 & 1 & 1 & 1\\ 
\end{array}
\quad
\begin{array}{r|rrr}
s & 1 & 2 & 3 \\ \hline
1 & 1 & 0 & 2 \\
2 & 0 & 1 & 0 \\
3 & 2 & 0 & 1
\end{array}
\quad 
\begin{array}{r|rrr}
r & 1 & 2 & 3 \\ \hline
1 & 2 & 2 & 2 \\
2 & 2 & 2 & 2 \\
3 & 2 & 2 & 2
\end{array}
\]
over $R$. We then compute (via \texttt{python}) the values of the natural path
polynomial for each of the oriented surface-links in the table at \cite{Y}:
\[
\begin{array}{r|l}
L & \Phi_{\vec{D}}^{MP}(L)  \\ \hline
2_1 &  4xy^2 \\ 
6^{0,1}_1 &  2x^2y^2 + 6xy^2 \\ 
8_1 & 4x^2y^2 \\ 
8^{1,1}_1 & 6xy^2 \\ 
9_1 &  4x^2y^2 \\ 
9^{0,1}_1 & 4x^2y^2 + 6xy^2 \\ 
10_1 & 4xy^2 \\
\end{array}
\quad
\begin{array}{r|l}
L & \Phi_{\vec{D}}^{MP}(L)  \\ \hline
10_2 & 4x^2y^2 \\ 
10_3 & 4xy^2 \\
10^1_1 & 4x^2y^2 \\ 
10^{0,1}_1 & 4x^2y^2+6xy^2 \\ 
10^{0,1}_2 & 8x^2y^2+2xy^2 \\ 
10^{1,1}_1 & 6xy^2 \\
10^{0,0,1}_1 & 6x^3y^2+14x^2y^2.
\end{array}
\]
\end{example}

\section{\textbf{Questions}}\label{Q}

We end with some questions and directions for future research. 

The main question is how to interpret these invariants -- what is the 
geometric meaning of the natural path polynomial?

What other quiver representations can be defined on the biquandle module
quiver? What other decategorifications are possible? Is it always possible to
find a biquandle module quiver representation distinguishing any two 
non-equivalent knots, virtual knots, links, or surface-links?

\bibliography{yj-sn3}{}
\bibliographystyle{abbrv}

\bigskip

\noindent
\textsc{Department of Mathematics \&Research Institute for Natural Science\\
Hanyang University\\
Seoul 04763, Republic of Korea}

\medskip

\noindent
\textsc{Department of Mathematical Sciences \\
Claremont McKenna College \\
850 Columbia Ave. \\
Claremont,  CA 91711}

\end{document}